\documentclass[a4paper,oneside,reqno,10pt]{amsart}
\usepackage[latin1]{inputenc}
\usepackage{amsmath, amssymb,amsthm}
  \usepackage{paralist}
  \usepackage{graphics} 
  \usepackage{epsfig} 
\usepackage{graphicx}  \usepackage{epstopdf}
 \usepackage[colorlinks=true]{hyperref}
\hypersetup{urlcolor=blue, citecolor=red}

\usepackage{placeins}
\usepackage{enumerate}

\usepackage{comment}
 
\usepackage[
backend=bibtex,
style=alphabetic,
sorting=ynt
]{biblatex}
\newtheorem{theorem}{Theorem}[section]

\newtheorem{lemma}[theorem]{Lemma}

\theoremstyle{definition}
\newtheorem{definition}[theorem]{Definition}
\newtheorem{remark}{Remark}

\newcommand{\rr}{\mathbf{R}}

\newcommand{\eff}{{\mathrm{eff}}}
\newcommand{\di}[1]{{\mathrm{div}}\big( #1\big)}

\newcommand{\disp}{\displaystyle}
\newcommand{\ve}{\varepsilon}

\newcommand{\ii}[1]{\int_{#1}}
\newcommand{\Oe}{\Omega_\varepsilon}
\newcommand{\Se}{\Sigma_\varepsilon}
\newcommand{\Ge}{\Gamma_\varepsilon^\pm}

\newcommand{\be}{\begin{equation}}
\newcommand{\ee}{\end{equation}}
\newcommand{\ba}{\begin{array}}
\newcommand{\ea}{\end{array}}

\title[]{Two-scale convergence in thin domains with locally periodic rapidly oscillating boundary}


\thanks{}

\makeatletter
\g@addto@macro{\endabstract}{\@setabstract}
\newcommand{\authorfootnotes}{\renewcommand\thefootnote{\@fnsymbol\c@footnote}}%
\makeatother

\begin{document}
\maketitle

\begin{center}

  \normalsize
  \authorfootnotes
  Irina Pettersson \textsuperscript{1}
  \par \bigskip
  \textsuperscript{1}UiT The Arctic University of Norway \par
  \bigskip

\today
\end{center}

\begin{abstract}
The aim of this paper is to adapt the notion of two-scale convergence in $L^p$ to the case of a measure converging to a singular one. We present a specific case when a thin cylinder with locally periodic rapidly oscillating boundary shrinks to a segment, and the corresponding measure charging the cylinder converges to a one-dimensional Lebegues measure of an interval. The method is then applied to the asymptotic analysis of  linear elliptic operators with locally periodic coefficients in a thin cylinder with locally periodic rapidly varying thickness.

\end{abstract}

\section{Introduction}
The goal of this paper is twofold. 
First, we want to adapt the classical two-scale convergence (see \cite{Ng-1989}, \cite{Al-1992}, \cite{Zh-2000}) to the case of a asymptotically thin domain. We consider a specific case when the domain has locally periodic rapidly oscillating boundary and shrinks to a segment. Second, we will apply the introduced definition to the asymptotic analysis of a linear elliptic operator with locally periodic coefficients in a thin domain with oscillating thickness.  

The two-scale convergence is a powerful tool that allows us to characterise the leading term of the asymptotics without using asymptotic expansions, that reduces the amount of computations. It can be applied both to linear and nonlinear problems, which makes this method so popular for asymptotic analysis.  In \cite{maruvsic2000two} the authors introduced the notion of the two-scale convergence for thin domains, but their definition does not catch the oscillations in the longitudinal variable. As a consequence, it works for operators with coefficients which are constant in the longitudinal variable.

Boundary value and spectral problems in thin domains are usually treated using the analysis of resolvents (\cite{FrSo-2009}), the method of asymptotic expansions (see for example \cite{ciarlet1979justification}, \cite{Panas-05}, \cite{borisov2010asymptotics}, \cite{MePo-2010}, \cite{Naz-01}, \cite{PeSi-2013}), two-scale convergence (\cite{EnePau-1996}, \cite{maruvsic2000two}, \cite{PaPi-2011}, \cite{PaPe-2015}), $\Gamma$-convergence (\cite{MuSi-1995}, \cite{ansini2001homogenization}, \cite{BrFoFr-2000}, \cite{GaGu-2002}, \cite{bouchitte2007curvature}, \cite{BoMaTr-12}), compensated compactness agrument (\cite{GuMo-2003}), and the unfolding method (\cite{blanchard2008microscopic}, \cite{ArPe-2011}, \cite{ArVi-2017}).  The presented list of works devoted to the homogenization in thin structures is far from being complete, but our primary focus is the case of thin domains with locally periodic rapidly varying thickness, and to our best knowledge the works closely related to our study are \cite{MePo-2010}, \cite{ArPe-2011}, \cite{FrSo-2009}, \cite{borisov2010asymptotics}, and \cite{NaPeTa-2016}. We describe them briefly below. 

The case of periodic rapidly oscillating boundary was considered in \cite{MePo-2010}, where the authors studied the asymptotic behaviour of second-order self-adjoint elliptic operators with periodic coefficients, for different boundary conditions. In \cite{ArPe-2011} the case of a locally periodic rapidly oscillating boundary was addressed, and the authors studied the Neumann boundary value problem for the Laplace operator in a two-dimensional thin domain by means of the unfolding method. 
Spectral asymptotics of the Laplace operator in thin domains with slowly varying thickness were considered in \cite{FrSo-2009}, \cite{borisov2010asymptotics}, \cite{NaPeTa-2016}, where under the Dirichlet boundary conditions the localization of eigenfunctions occur. 

The contribution of the present paper is an adapted notion of the two-scale convergence that covers both thin domains with slowly varying, periodic rapidly oscillating and locally periodic rapidly oscillating boundary. We do not make any restrictions on the dimension of the thin domains in the transverse direction. The method presented can be applied to both boundary value and spectral problems (exactly like the classical two-scale convergence), linear and nonlinear. In the present note we use it for the homogenization of a linear elliptic operator with locally periodic coefficients in a thin domain with locally periodic rapidly oscillating boundary. Our approach is based on the two-scale convergence in spaces with measure introduced in \cite{BouFra-01}, \cite{Zh-2000}. It was introduced for the case of a scaled periodic measure, while in the present work we focus on a measure converging to a singular one. The proofs of the basic facts about the properties of the $L^p$-spaces and the two-scale convergence itself follow the lines of those in  \cite{Zh-2000}. 

The paper is organized as follows. In Section \ref{spaces} we define the domain and introduce the corresponding spaces with measure charging this domain. In Section \ref{two-scale} we introduce the adapted two-scale convergence and discuss its properties. Section \ref{linear-case} concerns with the application of the method to the asymptotic analysis of a linear elliptic operator with locally periodic coefficients (see Theorem \ref{th:main}). 


\section{Variable spaces with singular measure in a cylinder with locally periodic rapidly oscillating boundary}
\label{spaces}

We are going to adapt the notion of the two-scale convergence to the case when a thin domain has a rapidly oscillating boundary modulated by some (slowly) varying function.

In what follows the points in $\rr^d$ are denoted by $x=(x_1, x')$,  and $I=(-L,L)$. 
We denote
\begin{align*}
Q(x_1, y_1)= \{y'\in\rr^{d-1}: F(x_1,y_1,y')>0\},
\end{align*}
where $F(x_1,y_1,y')$ satisfies the conditions
\begin{itemize}
\item[(F1)]
$F(x_1, y_1, y') \in C^{1, \alpha}(\overline I \times \overline I \times \rr^{d-1})$ is periodic with respect to $y_1$.

\item[(F2)]
$F + |\nabla_{y}F|\neq 0$, that is $F$ cannot have maximum/minimum where it vanishes.
\item[(F3)]
$F\big|_{y_1=0}=1$, $F\big|_{\pm L} \le 0$.
\item[(F4)]
$Q(x_1, y)$ is simply connected.
\end{itemize}


Now let $\varepsilon>0$ be a small parameter.
We are going to work in a thin cylinder 
\begin{align*}
\Oe = \{x=(x_1,x'): x_1 \in I, x' \in \ve Q(x_1,\frac{x_1}{\ve})\}.
\end{align*}

An example of $\Oe$ is presented in Figure \ref{fig:rod} for three different values of $\ve$. 

\begin{figure}
\includegraphics[trim = 1in 3in 1in 3in, clip, width=0.8\textwidth]{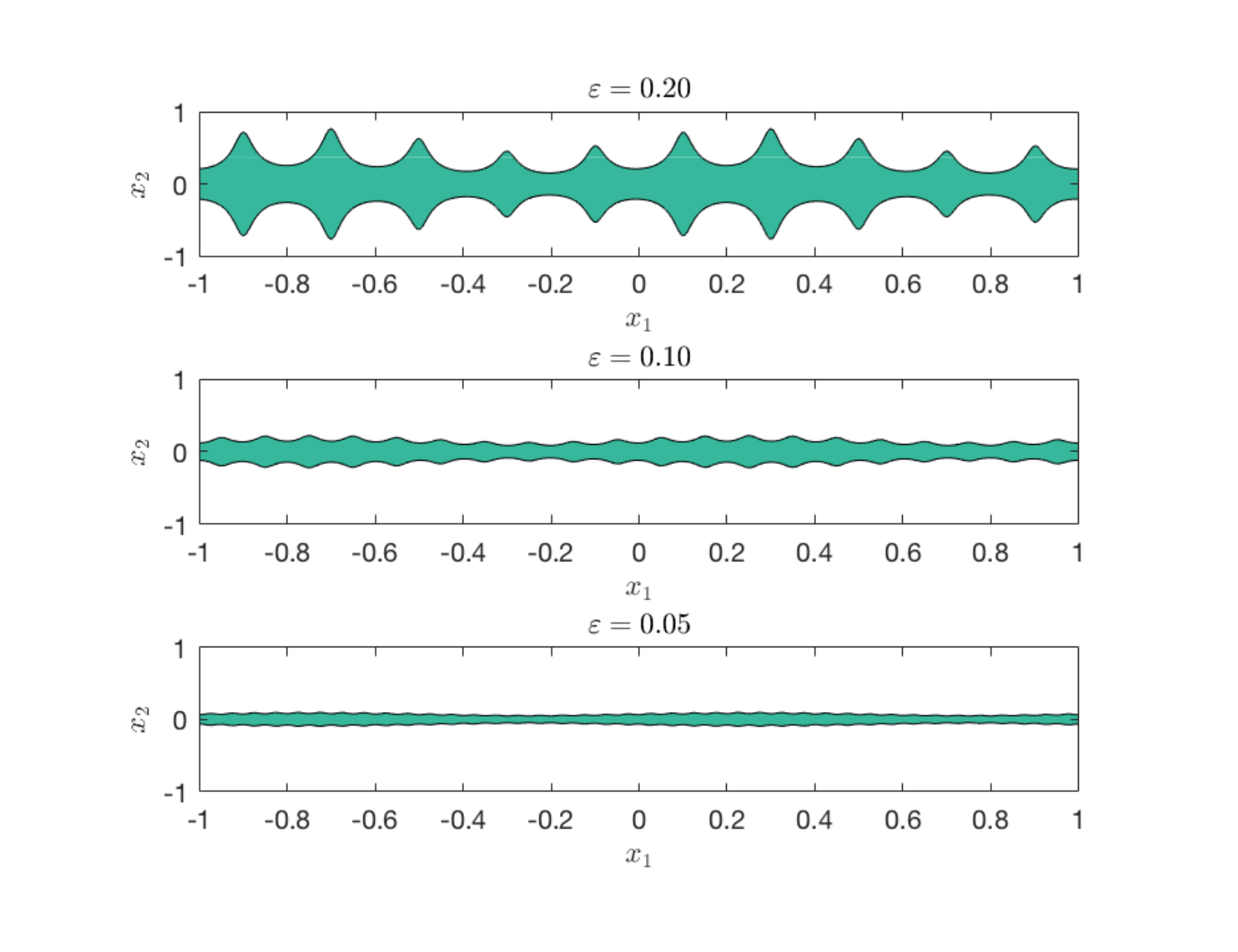}
\caption[caption]{\small{Thin cylinder generated by\\ \mbox{$F(x_1, y_1, y_2)=2+\sin(2\pi x_1)-y_2^2\cdot (1+4\ve\cos(2\pi y_1))$}.}}
\label{fig:rod}
\end{figure}

Here $Q(x_1,\frac{x_1}{\ve})$ describes the locally periodically varying cross section of the cylinder (periodicity with respect to the second variable is inherited from $F$).
The boundary of $\Oe$ consists of the lateral boundary of the cylinder 
\begin{align*}
\Se = \{x=(x_1,x'): x_1 \in I,  F(x_1,\frac{x_1}{\ve},\frac{x'}{\ve})=0\},
\end{align*}
and the bases $\Ge=\{\pm L\} \times (\ve Q(\pm L, \pm L/\ve))$.

The periodicity cell depending on $x_1$ is
\begin{align*}
\Box(x_1) = \{y=(y_1,y'): y_1 \in \mathbf T^1, y' \in Q(x_1,y_1)\},
\end{align*}
where $\mathbf T^1$ is a one-dimensional torus.

Since $F(x_1,y_1, y')$ is periodic in $y_1$, the boundary of $\Box(x_1)$ is $\partial \Box(x_1) = \{y=(y_1,y'): y_1 \in \mathbf T^1, F(x_1,y_1,y')=0\}$.








We define a Radon measure on $\rr^d$ by
\begin{equation}
\label{def:measure}
d \mu_\ve =  \ve^{-(d-1)} \chi_{\Oe}(x) \, dx, 
\end{equation}
where $\chi_{\Oe}(x)$ is the characteristic function of the thin cylinder $\Oe$; $dx$ is the $d$-dimensional Lebesgue measure.

The factor  $\ve^{-(d-1)}$ in \eqref{def:measure} makes the measure of the cylinder $\Oe$ of order $1$.

\begin{lemma}
\label{lm:conv-measures}
The $\mu_\ve$ defined by (\ref{def:measure}) converges weakly, as $\ve \to 0$,  to the measure $\mu_\ast$ defined by
$$
d \mu_\ast= |\Box(x_1)| \chi_I(x_1)  dx_1 \times \delta(x').
$$ 
\end{lemma}

\begin{proof}
Let $\varphi \in C_0(\rr^d)$.
Then
\begin{align*}
 \int_{\rr^d} \varphi(x) \, d\mu_\ve(x)
 =
\int_{I} \ve^{-(d-1)}\int_{\ve Q(x_1,x_1/\ve)} \varphi(x) dx' dx_1.
\end{align*}
Rescaling $y'=x'/\ve$ gives
\begin{align*}
 \int_{\rr^d} \varphi(x) \, d\mu_\ve(x)
 =
\int_{I} \int_{Q(x_1,x_1/\ve)} \varphi(x_1,\ve y') dy' dx_1.
\end{align*}
Let us divide the interval $I$ into small subintervals (translated periods) $I_j^\ve = \ve [0,1)+\ve j$, $j\in \mathbf Z$. On each such interval we use the mean-value theorem choosing a point $\xi_j$ and get
\begin{align*}
\sum_{j} \int_{I_j^\ve} \int_{Q(x_1,x_1/\ve)} \varphi(x_1,\ve y') dy' dx_1
=\sum_{j} \int_{I_j^\ve} \int_{Q(\xi_j,x_1/\ve)} \varphi(\xi_j,\ve y') dy' dx_1. 
\end{align*}
Since $Q(x_1, y_1)$ is periodic with respect to $y_1$, rescaling $y_1=x_1\ve$ yeilds
\begin{align*}
\sum_{j} \int_{\mathbf T^1} \int_{Q(\xi_1,y_1)} \varphi(\xi_j,\ve y') dy' dy_1
=
\sum_{j}  \ve\int_{\Box(\xi_j)} \varphi(\xi_j,\ve y') dy.
\end{align*}
The last sum is a Riemann sum converging, as $\ve \to 0$, to the following integral
\begin{align*}
\sum_{j}  \ve\int_{\Box(\xi_j)} \varphi(\xi_j,\ve y') dy &\to 
\ii{I}  \ii{\Box(x_1)} \varphi(x_1,0) \,dy dx_1 \\
& = \ii{I} |\Box(x_1)| \varphi(x_1,0) = \ii{\rr^d} \varphi(x) \, d\mu_\ast.
\end{align*}

Note that, for any $x_1 \in I$, due to the continuity of $F$, $|x'| \le C \ve^{d-1}$. Given $\gamma > 0$, we can choose $\ve$ small enough such that $x' \in \ve Q$ implies $|\varphi(x_1,0) - \varphi(x)| < \gamma$ using the uniform continuity of $\varphi$.
\end{proof}

\begin{remark}
We assume that the cylinder is bounded, but all the argument apply to the case when it grows in the $x_1$ direction, as $\ve \to 0$. The arguments are valid if the cylinder has uniformly bounded thickness. In the case of a cylinder growing in $x_1$, as $\ve \to 0$, the limit measure is $d\mu_\ast = |\Box(x_1)|dx_1 \times \delta(x')$.
\end{remark}
\begin{remark}
Note that the geometry of the boundary of the periodicity cell is of no importance in Lemma \ref{lm:conv-measures}.
\end{remark}

For any $\ve$ and $1<p<\infty$, the space of Borel measurable functions $g: \rr^d \to \rr$ such that
\[
\int_{\rr^d} |g|^p\, d\mu_\ve < \infty,
\]
is denoted by $L^p(\rr^d, \mu_\ve)$. For vector functions $g: \rr^d \to \rr^d$ we denote the corresponding space by $L^p(\rr^d, \mu_\ve)^d$.

\begin{definition}
A sequence $u_\ve$ is bounded in $L^p(\rr^d,\mu_\ve)$ if
\begin{align*}
\limsup \limits_{\ve \to 0} \ii{\rr^d} |u_\ve|^p d\mu_\ve < \infty.
\end{align*}
A bounded sequence $u_\ve \in L^p(\rr^d, \mu_\ve)$ is said to converge weakly in $L^p(\rr^d,\mu_\ve)$ to $u \in L^p(\rr^d,\mu_\ast)$ if
\begin{align*}
\lim \limits_{\ve \to 0} \ii{\rr^d} u_\ve \varphi \, d\mu_\ve =  \ii{\rr^d} u \varphi \, d\mu_\ast, \quad \varphi \in C_0^\infty(\rr^d).
\end{align*} 
We say that  $u_\ve \in L^p(\rr^d, \mu_\ve)$ converges strongly to $u \in L^p(\rr^d,\mu_\ast)$ if for any  $v_\ve \in L^{p'}(\rr^d, \mu_\ve)$ weakly converging to  $v \in L^{p'}(\rr^d,\mu_\ast)$, $1/p +1/p'=1$, we have
\begin{align*}
\lim \limits_{\ve \to 0} \ii{\rr^d} u_\ve \, v_\ve \, d\mu_\ve =  \ii{\rr^d} u \, v \, d\mu_\ast.
\end{align*}

\end{definition}

Proofs of the following facts valid for a sequence of measures $\mu_\ve$ weakly convergent to $\mu_\ast$ (no specific assumptions on the structure of $\mu_\ve$), can be found in \cite{Zh-2004}.
\begin{itemize}
\item

The property of weak compactness of a bounded sequence in a separable Hilbert space remains valid with respect to the convergence in variable spaces.
Any bounded sequence in $L^p(\rr^d,\mu_\ve)$ contains a weakly convergent subsequence.
\item
For $u_\ve \in L^p(\rr^d, \mu_\ve)$ weakly converging to $u \in L^p(\rr^d, \mu_\ast)$ the lower semicontinuity property holds:
\begin{align*}
\liminf \limits_{\ve \to 0} \ii{\rr^d} |u_\ve|^p d\mu_\ve \ge \ii{\rr^d} |u|^p d\mu_\ast.
\end{align*}

\item
A sequence $u_\ve \in L^p(\rr^d, \mu_\ve)$ converges strongly to $u \in L^p(\rr^d,\mu_\ast)$ if and only if $u_\ve$ converges to $u$ weakly and
\begin{align*}
\lim\limits_{\ve \to 0} \ii{\rr^d}|u_\ve|^p d\mu_\ve = \ii{\rr^d}|u|^p d\mu_\ast.
\end{align*}  
\end{itemize}

Let us also recall the definition of the Sobolev space with measure.
\begin{definition}
A function $g\in L^p(\rr^d, \mu_\ve)$ is said to belong to the space $W^{1,p}(\rr^d, \mu_\ve)$ if there exists a vector function $z \in L^p(\rr^d, \mu_\ve)^d$ and a sequence $\varphi_k\in C_0^\infty(\rr^d)$ such that
\[
\varphi_k \to g \quad \mbox{in} \,\, L^p(\rr^d,\mu_\ve), \quad k \to \infty,
\]
\[
\nabla \varphi_k \to z \quad \mbox{in} \,\, L^p(\rr^d,\mu_\ve)^d, \quad k \to \infty.
\]
In this case $z$ is called a gradient of $g$ and is denoted by $\nabla^{\mu_\ve} g$.
\end{definition}
Since in our case the measure $\mu_\ve$ is a weighted Lebesgue measure, we have $\nabla^{\mu_\ve}g =\nabla g$ and
the space $W^{1,p}(\rr^d, \mu_\ve)$ is identical to the usual Sobolev space $W^{1,p}(\Oe)$, in contrast to the scaled periodic singular measure considered in \cite{Zh-2000} when the gradient is not unique and is defined up to a gradient of zero.

The spaces $L^2(\rr^d, \mu_\ast)$ and $W^{1,p}(\rr^d, \mu_\ast)$ are defined in a similar way, however the $\mu_\ast$-gradient is not unique and is defined up to a gradient of zero. A zero function might have a nontrivial gradient as it is demostrated by Example 1 in Ch. 3, \cite{Zh-2000}. Following the proof in the last example, one can see that for $p=2$ the subspace of vectors of the form $(0, \psi_2(z_1), \ldots, \psi_d(z_1))$, $\psi_j \in L^2(\mathbf R)$ is the subspace of gradients of zero $\Gamma_{\mu_\ast}(0)$. Any $\mu_\ast$-gradient of $v \in W^{1,2}(\mathbf R^d, \mu_\ast)$ takes the form
\[
\nabla^{\mu_\ast} v(z) = (v'(z_1, 0), \psi_2(z_1), \ldots, \psi_d(z_1)), \quad \psi_j \in L^2(\mathbf R),
\]
where $v'(z_1, 0)$ is the derivative of the restriction of $v(z)$ to $\mathbf R$.


\section{Two-scale convergence in spaces with measure converging to a singular one}
\label{two-scale}
In what follows $\mu_\ve$ denotes the measure given by
\begin{align*}
d \mu_\ve = \chi_{\Oe}(x) \ve^{-(d-1)} \, dx, 
\end{align*}
and $\mu_\ast = |\Box(x_1)| \chi_I(x_1) \, dx_1 \times \delta(x')$ is the limit measure. 

For each $x_1 \in I$, we introduce $C^k(\Box(x_1))$, $L^p(\Box(x_1))$ and $W^{1,p}(\Box(x_1))$ in a usual way. Functions belonging to this spaces are $1$-periodic with respect to $y_1$. 
 
In the present context two-scale convergence is described as follows.
\begin{definition}
We say that $g^\ve\in L^p(\rr^d, \mu_\ve)$, $1<p<\infty$, converges two-scale weakly, as $\ve \to 0$, in $L^p(\rr^d, \mu_\ve)$ if
\begin{enumerate}[(i)]
\item $\limsup_{\ve \to 0} \|g^\ve\|_{L^p(\rr^d, \,\mu_\ve)} \le C$,
\item[$(ii)$]
there exists a function
${g}(x_1,y) \in L^p(I; L^p(\Box(x_1))$ $1$-periodic in $y_1$  such that the following limit relation holds:
\begin{align*}
\lim \limits_{\ve \to 0} \int_{\mathbf{R}^d} g^\ve(x) \, \varphi(x)\, \psi(\frac{x}{\ve}) d\mu_\ve(x) &=
\int_{\mathbf{R}^d} \frac{1}{|\Box(x_1)|}\int_{\Box(x_1)}  {g}(x_1,y)\, \varphi(x)\, \psi(y) \, dy\, d\mu_\ast(x)\\
&=
\int_{\mathbf{R}} \int_{\Box(x_1)} {g}(x_1,y)\, \varphi(x_1,0)\, \psi(y) \, dy\, dx_1,
\end{align*}
for any $\varphi\in C_0^\infty(\mathbf{R}^d)$ and $\psi(y)\in C^\infty(\overline{\Box(x_1)})$ periodic in $y_1$.
\end{enumerate}
We write $g^\ve \overset{2}{\rightharpoonup} g(x_1, y)$ if $g^\ve$ converges two-scale weakly to $g(x_1,y)$ in $L^p(\mathbf{R}^d, \mu_\ve)$. 


\vskip 0.5cm

The definition of the two-scale convergence holds for more general classes of test functions. 
Following the lines of the proof of Lemma \ref{lm:conv-measures} one can see that for $\psi(y) \in L^1(\Box(x_1))$ we have the mean-value property
\begin{align*}
\label{eq:mean-value-property}
\lim \limits_{\ve \to 0} \int_{\mathbf{R}^d} \varphi(x)\, \psi(\frac{x}{\ve}) d\mu_\ve(x) &=
\int_{\mathbf{R}^d} \frac{1}{|\Box(x_1)|}\int_{\Box(x_1)}  \varphi(x)\, \psi(y) \, dy\, d\mu_\ast(x)\\
&=
\int_{\mathbf{R}}  \varphi(x_1,0)\, \Big(\int_{\Box(x_1)} \psi(y) \, dy\Big)\, dx_1.
\end{align*}

For example, as it is shown in Lemma 3.1 in \cite{Zh-2004}, one can take a Caratheodory function $\Phi(x, y)$ such that
\begin{align*}
|\Phi(x,y)|\le \Phi_0(y), \quad \Phi_0 \in L^1({\Box(x_1)}).
\end{align*}
Such test functions are called admissible, and the mean-value property holds
\begin{align*}
\lim \limits_{\ve \to 0} \ii{\rr^d} \Phi(x, \frac{x}{\ve}) d\mu_\ve &= \ii{\rr^d} \frac{1}{|\Box(x_1)|}\int_{\Box(x_1)}  \Phi(x,y) dy d\mu_\ast\\
&= \ii{\rr}\int_{\Box(x_1)}  \Phi(x_1,0,y) dy dx_1.
\end{align*}
The proof of the mean-value property follows the lines of the proof of Lemma 3.1 in \cite{Zh-2004}.
As it was shown in \cite{Al-1992}, the property of continuity with respect to one of the arguments can not be dropped.

\end{definition}

The following compactness result can be proved in the same way as Theorem 4.2 in \cite{Zh-2004}.

\begin{lemma}[Compactness]
\label{lm:compactness-1}
Suppose that $g^\ve$ satisfies the estimate
\[
\limsup_{\ve \to 0}\|g^\ve\|_{L^p(\mathbf{R}^d,\, \mu_\ve)} \le C.
\]
Then $g^\ve$, up to a subsequence, converges two-scale weakly in $L^p(\mathbf{R}^d, \mu_\ve)$ to some function ${g}(x_1,y) \in L^p(\mathbf{R}^d \times \Box(x_1), \mu_\ast \times dy)$.
\end{lemma}
\begin{definition}
A sequence $g^\ve$ is said to converge two-scale strongly to a function ${g}(x_1, y) \in L^p(\mathbf{R}^d \times \Box(x_1), \mu_\ast \times dy)$ if
\begin{enumerate}[(i)]
\item $g^\ve$ converges two-scale weakly to ${g}(x_1, y)$,
\item the following limit relation holds:
\[
\lim \limits_{\ve \to 0} \int_{\mathbf{R}^d} |g^\ve(x)|^p d\mu_\ve(x) =
\int_{\mathbf{R}^d} \frac{1}{|\Box(x_1)|}\int_{\Box(x_1)}  |{g}(x_1,y)|^p\, dy\, d\mu_\ast(x).
\]
We write $g^\ve \overset{2}{\rightarrow} g(x_1, y)$ if $g^\ve$ converges two-scale strongly to the function $g(x_1,y)$ in $L^p(\mathbf{R}^d, \mu_\ve)$.
\end{enumerate}
\end{definition}

The following properties of the weak two-scale limit hold (see \cite{Zh-2004} for the proof in spaces with measure):
\begin{itemize}
\item
If $u_\ve\overset{2}{\rightharpoonup} u(x_1,y)$ in $L^p(\rr^d, \mu_\ve)$, then $u_\ve$ converges weakly in $L^p(\rr^d, \mu_\ve)$ to the local average of the two-scale limit:
\begin{align*}
u_\ve \rightharpoonup \frac{1}{|\Box(x_1)|}\int_{\Box(x_1)}  u(x_1, y) dy.
\end{align*} 
To see this it is suffices to take a test function independent of $y$ in the definition of the two-scale convergence.

\item
If $u_\ve\overset{2}{\rightharpoonup} u(x_1,y)$ in $L^p(\rr^d, \mu_\ve)$, then the lower semicontinuity property holds
\begin{align*}
\liminf \limits_{\ve \to 0} \ii{\rr^d}|u_\ve|^p d\mu_\ve &\ge \ii{\rr^d} \frac{1}{|\Box(x_1)|}\int_{\Box(x_1)} |u(x_1,y)|^p dy d\mu_\ast \\
& = \ii{\rr} \int_{\Box(x_1)} |u(x_1,y)|^p dy dx_1.
\end{align*}
A proof is based on the Young inequality
\begin{align*}
a\cdot b \le \frac{1}{p}|a|^p + \frac{1}{p'}|b|^{p'}, \quad \frac{1}{p} + \frac{1}{p'}=1.
\end{align*}
For any $\varphi(x_1,y) \in C_0^\infty(\rr; C^\infty(\Box(x_1)))$
\begin{align*}
\frac{1}{p}\ii{\rr^d} |u_\ve|^p d\mu_\ve \ge \ii{\rr^d} u_\ve \varphi(x_1,\frac{x}{\ve})dy d\mu_\ve
-  \frac{1}{p'}\ii{\rr^d} |\varphi(x_1,\frac{x}{\ve})|^{p'} d\mu_\ve.
\end{align*}
Passing to the limit yields
\begin{align*}
\frac{1}{p}\liminf \limits_{\ve \to 0} \ii{\rr^d} |u_\ve|^p d\mu_\ve
& \ge
 \ii{\rr^d} \frac{1}{|\Box(x_1)|}\int_{\Box(x_1)}  u(x_1,y)\varphi(x_1,y)dy d\mu_\ast \\
&- \frac{1}{p'}\ii{\rr^d}\frac{1}{|\Box(x_1)|}\int_{\Box(x_1)}  |\varphi(x_1,y)|^{p'} dy d\mu_\ast.
\end{align*}
By density of smooth functions in $L^p(\rr^d, \mu_\ve)$, we can take $\varphi(x_1,y) = |u(x_1,y)|^{p-2} u(x_1,y)$, which completes the proof.
\end{itemize}

The next proposition provides additional information about the two-scale limit in the case when it is possible to estimate the derivatives. The original statement is given for a fixed domain $\Omega$ and a fixed Lebegue measure in \cite{Al-1992} (Proposition 1.14). The case of a periodic scaled measure $\mu_\ve$ is considered in \cite{ChPiSh-07} (Theorem 10.3). The proof is essentially the same in all these cases and is therefore omitted.

\begin{lemma}
\label{lm:two-scale}
Assume that $u_\ve(x)$ is bounded in $W^{1,p}(\rr^d, \mu_\ve)$, $1\le p < \infty$, and
\begin{align*}
\ii{\rr^d}|u_\ve|^p d\mu_\ve +
\ii{\rr^d}|\nabla u_\ve|^p d\mu_\ve \le C.
\end{align*}
Then there exists $u(x_1) \in W^{1,p}(\rr^d, \mu_\ast)$ and $u_1(x_1,y) \in L^p(\rr; W^{1,p}(\Box(x_1)))$ periodic in $y_1$ such that, as $\ve \to 0$,
\begin{enumerate}[(i)]
\item
$u_\ve$ strongly in $L^p(\rr^d, \mu_\ve)$ and strongly two-scale in $L^p(\rr^d, \mu_\ve)$ converges to $u(x_1) \in L^p(\rr^d, \mu_\ast)$.
\medskip
\item
$\nabla u_\ve$, along a subsequence, weakly two-scale converges to $\nabla^{\mu_\ast} u(x_1) + \nabla_y u_1(x_1,y)$ in $L^p(\rr^d, \mu_\ve)$. 
Here $\nabla^{\mu_\ast} u(x_1)$ is one of the gradients (which are defined up to a gradient of zero) with respect to the measure $\mu_\ast$. 
\end{enumerate}

\end{lemma}

\section{Homogenization of a linear elliptic operator with locally periodic coefficients}
\label{linear-case}
Let us illustrate how one can apply the adapted notion of the two-scale convergence to the asymptotic analysis of a linear second-order elliptic operator with locally periodic coefficients stated in a thin domain with locally periodic rapidly oscillating boundary. Let the domain be that described in Section \ref{spaces}. To fix the ideas, let us consider the following boundary value problem
\begin{align}
\label{eq:linear-eq}
-\di{a^\ve \nabla u_\ve} + c^\ve u_\ve & = f, \quad \Oe,\nonumber \\
a^\ve \nabla u_\ve \cdot n &= 0, \quad \Se,\\
u_\ve &=0, \nonumber \quad \Ge.
\end{align} 
Our main assumptions are
\begin{itemize}
 \item[(H1)]
 The coefficients have the form $a^\ve(x) = a(x_1, \frac{x}{\ve})$, $c^\ve(x) = c(x_1, \frac{x}{\ve})$, where $c(x_1,y), a_{ij}(x_1, y) \in C^{1, \alpha}(\overline I; C^\alpha(\overline \Box(x_1)))$ are $1$-periodic in $y_1$, $0<\alpha<1$.
 \item[(H2)]
 The matrix $a$ is symmetric and satisfies the uniform ellipticity condition: There exists $\Lambda_0>0$ such that for all $x_1 \in I$ and $y \in \Box(x_1)$, 
 $$
 a_{ij}(x_1,y) \xi_i\xi_j \ge \Lambda_0 |\xi|^2, \quad \xi \in \rr^d.
 $$
 \item[(H3)]
$f(x_1) \in L^2(I)$.
\end{itemize}
We study the asymptotic behaviour of the solution $u_\ve$ of \eqref{eq:linear-eq} as $\ve \to 0$. 


Problem \eqref{eq:linear-eq} being stated in a bulk domain is classical and can be treated by any method of asymptotic analysis. We present the convergence result in the case when the domain is thin and has a locally periodic rapidly varying thickness using singular measures approach. Corrector terms, as well as the estimates for the rate of convergence can be obtained for example by using the asymptotic expansion method.  

\begin{theorem}
\label{th:main}
Let $u^\ve$ be a solution of problem \eqref{eq:linear-eq}.
Under the assumptions (H1)--(H3), the following convergence result holds:
\begin{enumerate}[(i)]
\item $u^\ve$ converges two-scale, as $\ve \to 0$, in $L^2(\rr^d, \mu_\ve)$ to a solution $u$ of the one-dimensional problem
\begin{align}
\label{eq:eff-prob-remark-1}
- (a^\eff(x_1) u')' + \bar{c}(x_1)\, u &=  |\Box(x_1)|\, f(x_1), \quad \hfill x \in (-L,L),
\\
u(\pm L)&=0.\nonumber
\end{align}
The effective diffusion coefficient  $a^\eff$ and the potential $\bar{c}$ are given by the formulae
\begin{align*}
a^\eff(x_1) & = \int_{\Box(x_1)} a_{1j}(x_1, y)(\delta_{1j} + \partial_{y_j} N_1(x_1, y))\, dy,\\
\bar{c}(x_1) & =  \int_{\Box(x_1)} c(x_1, y)\, dy.
\end{align*}
The auxiliary function $N_1(x_1, y)$ solves the following cell problem:
\begin{align*}
\left\{
\begin{array}{lcr}
\displaystyle
- {\rm{div}}_y (a(x_1,y)\nabla_y N_1(x_1, y)) = \partial_{y_i} a_{i 1}(x_1, y), \quad \hfill y \in \Box(x_1),
\\[2mm]
\displaystyle
a(x_1, y)\nabla_y N_1(x_1, y)\cdot n = - a_{i 1}(x_1, y)\, n_i, \quad \hfill y \in \partial \Box(x_1).
\end{array}
\right.
\end{align*}

\item
$
\disp
\lim \limits_{\ve \to 0}\frac{1}{\ve^{d-1}}\, \int_{\Omega_\ve} |u^\ve(x)-u(x_1)|^2\, dx = 0.
$
\item
As $\ve \to 0$, the corresponding fluxes converge two-scale in $L^2(\rr^d, \mu_\ve)$:
\begin{align*}
a^\ve(x) \nabla u^\ve  \overset{2}{\rightharpoonup} a^\eff(x_1) u'(x_1) e_1 + \nabla_y N(x_1,y) u'(x_1), \quad e_1 = (1,0,\cdots, 0) \in \rr^d.
\end{align*}

\end{enumerate}
\end{theorem}

\begin{proof}
The weak formulation of \eqref{eq:linear-eq} in terms of the measure $\mu_\ve$ reads
\begin{align}
\label{eq:var-formul-measures}
\ii{\rr^d} a^\ve \nabla u_\ve \cdot \nabla \Phi \, d\mu_\ve + \ii{\rr^d} c^\ve u_\ve \Phi \, d\mu_\ve = \ii{\rr^d} f \Phi \, d\mu_\ve, 
\end{align}
where $\quad \Phi \in H^1(\Oe), \Phi\big|_{\Ge}=0$. Taking $u_\ve$ as a test function we obtain the following a priori estimate:
\begin{align}
\label{eq:apriori-linear}
\|u_\ve\|_{L^2(\rr^d, \mu_\ve)} + \|\nabla u_\ve\|_{L^2(\rr^d, \mu_\ve)} \le C.
\end{align}
Thus, up to a subsequence,  $u_\ve$ converges two-scale weakly in $L^2(\rr^d, \mu_\ve)$ to some $u(x_1)\in L^2(\rr^d,\mu_\ast)$, and due to Lemma \ref{lm:two-scale}, there exists $u_1(x_1,y)\in L^2(\rr; H^1(\Box(x_1)))$ periodic in $y_1$ such that $\nabla u_\ve$ converges two-scale  in $L^2(\rr^d, \mu_\ve)$ to $\nabla^{\mu_\ast} u(x_1) + \nabla_y u_1(x_1,y)$.

We proceed in two steps.
First we choose an oscillating test function to determine the structure of $u_1(x_1,y)$.
Then we use a smooth test function of a slow argument to obtain the limit problem for $u$.

Let us take
\begin{align*}
\Phi_\ve(x) & = \ve\, \varphi(x)\, \psi(\frac{x}{\ve}), \quad \varphi \in C_0^\infty(\rr^d), \,\,\, \psi \in C^\infty(\mathbb T^1 \times \rr^{d-1}),
\end{align*}
as a test function in (\ref{eq:var-formul-measures}).

The gradient of $\Phi_\ve$ takes the form
\begin{align*}
\nabla \Phi_\ve(x) = \ve\, \psi(\frac{x}{\ve}) \nabla_x \varphi(x)  +
\varphi(x)\, \nabla_y \psi(y)\big|_{\zeta=x/\ve}.
\end{align*}
In the first term on the left hand side in  (\ref{eq:var-formul-measures}) we can regard $a^\ve$ as a part of the test function.
Passing to the limit we get
\begin{align*}
& \quad \int_{\mathbf{R}^d} \Big(\frac{1}{|\Box(x_1)|}  \int_{\Box(x_1)} a(x_1,y)\nabla_y \psi(y)\,dy\Big) \cdot \nabla^{\mu_\ast} u(x_1, 0)\, \varphi(x_1, 0) d\mu_\ast \\
& \qquad + \int_{\mathbf{R}^d} \Big(\frac{1}{|\Box(x_1)|}  \int_{\Box(x_1)} a(x_1,y)\nabla_y \psi(y)\cdot \nabla_y u_1(x_1, y)\,dy\Big) \, \varphi(x_1, 0) d\mu_\ast=0.
\end{align*}
\item


Looking for $u_1$ in the form
\be
\label{eq:v^1}
u_1(x_1, y)= N(x_1,y) \cdot \nabla^{\mu_\ast} u(x_1, 0)
\ee
gives the following relation for the components of $N(y)$:
\begin{align*}
& \int_{\mathbf{R}^d} \Big(\frac{1}{|\Box(x_1)|}  \int_{\Box(x_1)} a(x_1,y)\nabla_y N_k(y)\cdot \nabla \psi(y)\, dy\Big) \varphi(x_1, 0)\, d \mu_\ast\\
& \quad = - \int_{\mathbf{R}^d} \Big(\frac{1}{|\Box(x_1)|}  \int_{\Box(x_1)} a_{kj}(x_1,y)\partial_{y_j} \psi(y)\, dy\Big) \varphi(x_1, 0)\, d \mu_\ast,
\end{align*}
for any $\varphi \in C_0^\infty(\mathbf{R}^d)$, $\psi \in C^\infty(\mathbf T^1 \times \rr^{d-1})$.
The last integral identity is a variational formulation associated to
\be
\label{eq:N-k}
\left\{
\begin{array}{lcr}
\displaystyle
- {\rm{div}}_y(a(x_1,y)\nabla_y N_k(x_1, y)) = \partial_{y_i} a_{i k}(x_1, y), \quad \hfill y \in \Box(x_1),
\\[1.5mm]
\displaystyle
a(x_1, y)\nabla_y N_k(y)\cdot n = -  a_{i k}(x_1, y)\, n_i, \quad \hfill y \in \partial \Box(x_1),\quad k = 1, 2, \ldots.
\end{array}
\right.
\ee
For each $x_1 \in I$, there exists a unique solution
$N_k(x_1, \cdot) \in C^{1,\alpha}(I; C^{1,\alpha}\overline{ \Box(x_1)})/\mathbf R$ to~\eqref{eq:N-k}.

In this way
\begin{align*}
\nabla u_\ve \overset{2}{\rightharpoonup} (\nabla^{\mu_\ast} u(x_1, 0) + \nabla_y N(x_1,y)\cdot \nabla^{\mu_\ast} u(x_1, 0)),\quad \ve \to 0.
\end{align*}
Now the structure of the function $v^1(z_1, \zeta)$ is known, and we can proceed by deriving the problem for $u$.

We pass to the limit in the integral identity (\ref{eq:var-formul-measures}) with $\varphi(x) \in C_0^\infty(\mathbf R^d)$:
\begin{align*}
& \int_{\mathbf R^d} \Big( \frac{1}{|\Box(x_1)|}  \int_{\Box(x_1)} a(x_1, y)({\mathrm{Id}} + \nabla_y N(x_1, y))\, dy\Big) \nabla^{\mu_\ast} u(x_1, 0)\cdot \nabla \varphi(x_1, 0)\,d\mu_\ast
 \\
& \quad + \int_{\mathbf R^d} \frac{1}{|\Box(x_1)|}  \int_{\Box(x_1)}c(x_1,y) u(x_1, 0)\varphi(x_1, 0)\, dy d\mu_\ast \\
 & \qquad = \int_{\mathbf R^d} f(x_1,0)\, \varphi(x_1, 0)\, d\mu_\ast.
\end{align*}
 Here $\nabla N =\{\partial_{\zeta_i}N_j(\zeta)\}_{ij=1}^d$, and $\mathrm{Id} = \{\delta_{ij}\}_{ij=1}^d$ is the unit matrix.
Denote
\begin{align*}
A_{ij}^\eff =   \int_{\Box(x_1)} a_{ik}(x_1, y)(\delta_{kj} + \partial_{y_k}N_j(x_1,y))\, dy.
\end{align*}
In this way the limit problem in the weak form reads
\begin{align}
& \int_{\mathbf R^d} \frac{1}{|\Box(x_1)|} A^\eff \nabla^{\mu_\ast} u(x_1, 0)\cdot \nabla \varphi(x_1, 0)\,d\mu_\ast
 + \int_{\mathbf R^d} \frac{1}{|\Box(x_1)|} \overline c(x_1) u(x_1, 0)\varphi(x_1, 0)\, d \mu_\ast \notag
 \\
& \quad = \int_{\mathbf R^d} f(x_1,0)\, \varphi(x_1, 0)\, d\mu_\ast.
\label{eq:eff-weak-full}
\end{align}
The $\mu_\ast$-gradient is not unique, but the flux $A^\eff \nabla^{\mu_\ast} u(x_1, 0)$ is uniquely determined by the condition of orthogonality of the vector $A^\eff \nabla^{\mu_\ast} u$ to the subspace of the gradients of zero. This can be seen by taking in \eqref{eq:eff-weak-full} any test function with zero trace $\varphi(x_1, 0, \ldots,  0)=0$ and non-zero $\mu_\ast$-gradient, for example $\varphi(x)= \sum_{j\neq 1} x_j \psi_j(x_1)$ with arbitrary $\psi_j \in C_0^\infty(\mathbf R)\setminus \{0\}$. By the density of smooth functions, the subspace of vectors in the form $(0, \psi_2(x_1), \ldots, \psi_d(x_1))$, $\psi_j \in L^2(\mathbf R)$ is the subspace of the gradients of zero, and the condition of orthogonality to the gradients of zero gives that
$$
A^\eff \nabla^{\mu_\ast} u = (A_{1j}^\eff \partial_{x_j}^{\mu_\ast} u(x_1, 0), 0, \ldots, 0).
$$
If we define a solution of \eqref{eq:eff-weak-full} as a function $u(x) \in H^1(\mathbf R^d, \mu_\ast)$ satisfying the integral identity, then this solution is unique. A solution $(u, A^\eff \nabla^{\mu_\ast}u)$, as a pair,  is also unique due to the orthogonality to the gradients of zero. If one, however, defines a solution of \eqref{eq:eff-weak-full} as a pair $(u, \nabla^{\mu_\ast}u)$, then a solution is not unique. This has to do with the fact that the matrix $A^\eff$ is not positive definite, and the uniqueness of the flux does not imply the uniqueness of the gradient.

Next step is to prove that $A_{1j}^\eff = 0$ for all $j \neq 1$. To this end we rewrite the problem for $N_k$ in the following form:
\be
\label{eq:N-k-bis}
\left\{
\begin{array}{lcr}
\displaystyle
- {\rm{div}}_y(a(x_1,y)\nabla_y (N_k(x_1,y) + y_k) = 0, \quad \hfill y \in \Box(x_1),
\\
\displaystyle
a(x_1, y)\nabla_y (N_k(x_1,y) + y_k)\cdot n = 0, \,\,k = 1, 2, \ldots, \quad \hfill y \in \partial \Box(x_1).
\end{array}
\right.
\ee
We multiply \eqref{eq:N-k-bis} by $y_m$, $m \neq 1$, and integrate over $\Box(x_1)$. For $m \neq 1$, the function $y_m$ is periodic in $y_1$ and can be used as a test function. This gives
\begin{align*}
\int_{\Box(x_1)} a(x_1, y)\nabla_y (y_k + N_k(x_1,y)) \cdot \nabla y_m \, dy = 0,
\end{align*}
and since $\partial_{y_j} y_m = \delta_{jm}$, $A_{km}^\eff = 0$ for any $k=1, \ldots, d$ and $m \neq 1$.
Thus
\begin{align*}
A^\eff \nabla^{\mu_\ast} u = (A_{11}^\eff u'(x_1, 0), 0, \ldots, 0),
\end{align*}
and \eqref{eq:eff-weak-full} takes the form
\begin{align*}
& \int_{\mathbf R} A_{11}^\eff u'(x_1, 0) \varphi'(x_1, 0)\,dx_1
 + \int_{\mathbf R} \overline c(x_1) u(x_1, 0)\varphi(x_1, 0)\, d x_1 \notag
 \\
& \quad = \int_{\mathbf R} f(x_1,0)\, |\Box(x_1)|\varphi(x_1, 0)\, dx_1.
\end{align*}
Denoting $a^\eff = A_{11}^\eff$, $u(x_1)= u(x_1, 0)$, we see that the last integral identity is the weak formulation of (\ref{eq:eff-prob-remark-1}).

Using $N_i$ as a test function in \eqref{eq:N-k-bis} gives
\[
A_{ik}^\eff(x_1) = \int_{\Box(x_1)} a(x_1, \zeta)\nabla (y_i + \nabla_y N_i(x_1,y))\cdot \nabla_y (y_k + \nabla_y N_k(x_1,y))\, dy,
\]
which shows that $A^\eff$ is symmetric and positive semidefinite due to the corresponding properties of $a(x_1, y)$. If $e_1 = (1, 0,\ldots, 0)$,
\[
a^\eff = A_{11}^\eff = A^\eff e_1 \cdot e_1 \ge \Lambda_0 \int_{\Box(x_1)} |\nabla_y (y_1 + \nabla_y N_1(x_1,y))|^2\, dy.
\]
Assuming that $\partial_{y_i} (y_1 + \nabla_y N_1(x_1,y)) = 0$ for all $i$, leads to a contradiction since $N_1$ is periodic in $y_1$. Thus, the effective coefficient $a^\eff$ is strictly positive. 

It is left to prove the strong convergence of $u_\ve$ in $L^2(\Oe, \mu_\ve)$. To this end we consider the local average of $u_\ve$
\begin{align*}
\overline u_\ve(x_1) = \frac{1}{\ve^{d-1} |Q(x_1, x_1/\ve)|}\ii{\ve Q(x_1, x_1/\ve)} u_\ve(x) dx'.
\end{align*}
Applying the Poincar\'e inequality we obtain
\begin{align*}
\ii{\ve Q(x_1, x_1/\ve)} (u_\ve - \overline u_\ve)^2 dx' \le C \ve^2 \ii{\ve Q(x_1, x_1/\ve)} |\nabla(u_\ve - \overline u_\ve)|^2 dx'.
\end{align*} 
Integrating with respect to $x_1$, using (\ref{eq:apriori-linear}) and the definition of $\overline u_\ve$, we have
\begin{align}
\label{eq:1}
\ii{\Oe} (u_\ve - \overline u_\ve)^2 dx \le C\ve^2 \ii{\Oe} |\nabla(u_\ve - \overline u_\ve)|^2 dx \le C\ve.
\end{align}
At the same time, since $\overline u_\ve$ is bounded in $H^1(I)$, it converges strongly in $L^2(\rr^d, \mu_\ast)$ (equivalently in $L^2(I)$) to some $\overline u(x_1)$, which together with (\ref{eq:1}) gives the strong convergence of $u_\ve$ in $L^2(\Oe, \mu_\ve)$ to $u(x_1)=\overline u(x_1)$. 
\end{proof}

\printbibliography

\end{document}